\documentclass[a4paper,12pt]{amsart}

\usepackage{latexsym}
\usepackage{amssymb}
\usepackage{amsthm}
\usepackage{amscd}
\usepackage{amsmath}
\usepackage{graphics}
\usepackage[all]{xy}

\newtheorem{Theorem}{Theorem}[section]
\newtheorem{Lemma}[Theorem]{Lemma}
\newtheorem{Corollary}[Theorem]{Corollary}
\newtheorem{Proposition}[Theorem]{Proposition}

\newtheorem{Remarks}[Theorem]{Remarks}

\def\fka{{\mathfrak a}}
\def\fkb{{\mathfrak b}}
\def\fkp{{\mathfrak p}}

\def\fkm{{\mathfrak m}}

\def\opn#1#2{\def#1{\operatorname{#2}}}

\opn\Spec{Spec}
\opn\Supp{Supp}
\opn\supp{supp}
\opn\Max{Max}
\opn\max{max}
\opn\Min{Min}
\opn\min{min}
\opn\Ass{Ass}
\opn\Assh{Assh}
\opn\Ann{Ann}
\opn\depth{depth}
\opn\rank{rank}

\opn\Mat{Mat}

\opn\Tot{Tot}

\opn\Sym{Sym}
\def\Rees{{\mathcal R}}

\def\S{{\mathcal S}}
\def\ZZ{\mathbb{Z}}  %%% the ring of integers %%%

\opn\div{div}
\opn\Div{Div}
\opn\cl{cl}
\opn\Cl{Cl}

\opn\Ker{Ker}
\opn\Coker{Coker}
\opn\Im{Im}
\opn\Hom{Hom}
\opn\Tor{Tor}
\opn\Ext{Ext}
\opn\End{End}
\opn\Fitt{Fitt}
\opn\Aut{Aut}
\opn\id{id}

\opn\nat{nat}
\opn\pff{pf}%   \pf exists already
\opn\Pf{Pf}
\opn\GL{GL}
\opn\SL{SL}
\opn\G{G}
\opn\E{E}
\opn\H{H}
\opn\M{M}
\opn\mod{mod}
\opn\ord{ord}
\opn\det{det}
\opn\Soc{Soc}

\opn\chara{char}
\opn\length{\ell}
\opn\pd{pd}
\opn\rk{rk}
\opn\projdim{proj\,dim}
\opn\injdim{inj\,dim}
\opn\rank{rank}
\opn\depth{depth}
\opn\grade{grade}
\opn\height{ht}
\opn\embdim{emb\,dim}
\opn\codim{codim}

\renewcommand{\tilde}{\widetilde}

\title{A Note on the Buchsbaum-Rim function of a parameter module}

\author{Futoshi Hayasaka}
\author{Eero Hyry}

\address{Department of Mathematics, School of Science and Technology, 
Meiji University, 1-1-1 Higashimita, Tama-ku, Kawasaki 214--8571, JAPAN}
\email{hayasaka@isc.meiji.ac.jp}

\address{Department of Mathematics and Statistics, University of Tampere, 
33014 Tampereen yliopisto, FINLAND}
\email{Eero.Hyry@uta.fi}

\date{\today}


\begin{document}

\setlength{\baselineskip}{15pt}

%%%%%%%%%%%%%  ABSTRACT   %%%%%%%%%%%%%%%%

\begin{abstract}
In this article, we prove that the Buchsbaum-Rim function $\ell_A(\S_{\nu+1}(F)/N^{\nu+1})$ of a parameter module $N$ in $F$ is bounded above by $e(F/N) \binom{\nu+d+r-1}{d+r-1}$ for every integer $\nu \geq 0$. Moreover, it turns out that the base ring $A$ is Cohen-Macaulay once the equality holds for some integer $\nu$. As a direct consequence, we observe that the first Buchsbaum-Rim coefficient $e_1(F/N)$ of a parameter module $N$ is always non-positive. 
\end{abstract}

\maketitle

%%%%%%%%%%%%%%%%%%%%%%%%%%%%%%%%%%%%%%%%%%%%%%%%%%%%%
\section{Introduction}
%%%%%%%%%%%%%%%%%%%%%%%%%%%%%%%%%%%%%%%%%%%%%%%%%%%%%

Let $(A, \fkm)$ be a Noetherian local ring of dimension $d$. Let $F=A^r$ be a free module of rank $r>0$, and let $S=\S_A(F)$ be the symmetric algebra of $F$, which is a polynomial ring over $A$. For a submodule $M$ of $F$, let $\Rees(M)$ denote the image of the natural homomorphism  $\S_A(M) \to \S_A(F)$, which is a standard graded subalgebra of $S$. Assume that the quotient $F/M$ has finite length and $M \subseteq \fkm F$. Then we can consider the function $$ \lambda: \ZZ_{\geq 0} \to \ZZ_{\geq 0} \ ; \ \ \nu \mapsto \ell_A(S_{\nu+1}/M^{\nu+1})$$ where 
 $S_{\nu}$ and $M^{\nu}$ denote the homogeneous components of degree $\nu$ of $S$ and $\Rees(M)$, respectively. Buchsbaum and Rim studied this function in~\cite{BR2} in order to generalize the notion of the usual Hilbert-Samuel multiplicity of an $\fkm$-primary ideal. They proved that $\lambda(\nu)$ eventually coincides with a polynomial $P(\nu)$ of degree $d+r-1$. This polynomial can then be written in the form $$P(\nu)=\sum_{i=0}^{d+r-1}(-1)^i e_i(F/M)\binom{\nu+d+r-1-i}{d+r-1-i}$$ with integer coefficients $e_i(F/M)$. The coefficients $e_i(F/M)$ are called the {\it{Buchsbaum-Rim coefficients}} of $F/M$. The {\it{Buchsbaum-Rim multiplicity of}} $F/M$, denoted by $e(F/M)$, is now defined to be the leading coefficient $e_0(F/M)$. 

In their article Buchsbaum and Rim also introduced the notion of a parameter module (matrix), which generalizes the notion of a parameter ideal (system of parameters). The module $N$ in $F$ is said to be {\it{a parameter module in}} $F$, if the following three conditions are satisfied: (i) $F/N$ has finite length, (ii) $N \subseteq \fkm F$, and (iii) $\mu_A(N)=d+r-1$, where $\mu_A(N)$ is the minimal number of generators of $N$. 

A starting point of this article is the characterization of the Cohen-Macaulay property of $A$ given in \cite[Corollary 4.5]{BR2} by means of the equality $\ell_A(F/N)=e(F/N)$ for every parameter module $N$ of rank $r$ in $F=A^r$. Brennan, Ulrich and Vasconcelos observed in~\cite[Theorem 3.4]{BUV} that if $A$ is Cohen-Macaulay, then in fact $$\ell_A(S_{\nu+1}/N^{\nu+1}) = e(F/N)\binom{\nu+d+r-1}{d+r-1}$$ for all integers $\nu \geq 0$. Our main result is now as follows: 

\begin{Theorem}\label{main}
Let $(A, \fkm)$ be a Noetherian local ring of dimension $d>0$. 
\begin{enumerate}
\item[$(1)$] For any rank $r>0$, the inequality 
$$\ell_A(S_{\nu+1}/N^{\nu+1}) \geq e(F/N)\binom{\nu+d+r-1}{d+r-1} $$ 
always holds true for every parameter module $N$ in $F=A^r$ and for every integer $\nu \geq 0$. 
\item[$(2)$] The following statements are equivalent: 
\begin{enumerate}
\item[$(i)$] $A$ is a Cohen-Macaulay local ring; 
\item[$(ii)$] There exists an integer $r>0$ and a parameter module $N$ of rank $r$ in $F=A^r$ such that the equality $$\ell_A(S_{\nu+1}/N^{\nu+1}) = e(F/N)\binom{\nu+d+r-1}{ d+r-1} $$ holds true for 
some integer $\nu \geq 0$.
\end{enumerate}
\end{enumerate}
\end{Theorem}
This extends our previous result~\cite[Theorem 1.3]{HH2} where we assumed that $\nu=0$.
 
Concerning not only parameter modules but also their "powers", Theorem \ref{main} generalizes in two directions the classical result saying that the inequality $\ell_A(A/Q) \geq e(A/Q)$ holds true for any parameter ideal $Q$ in a local ring $A$ with equality for some parameter ideal if and only if $A$ is Cohen-Macaulay. 
Theorem \ref{main} seems to contain some new information even in the 
ideal case. Indeed, the equivalence of (i) and (ii) in (2) improves a recent observation that the ring $A$ is Cohen-Macaulay if there exists a parameter ideal $Q$ in $A$ such that 
$\ell_A(A/Q^{\nu+1})=e(A/Q) \binom{\nu+d}{d}$ for all $\nu \gg 0$ (see~\cite{GGHOPV, Ka}). Moreover, as a direct consequence of 
(1), we have the non-positivity of the first Buchsbaum-Rim coefficient of a parameter module. 

\begin{Corollary}\label{e1}
For any rank $r>0$, the inequality $$e_1(F/N) \leq 0$$ 
always holds true for every parameter module $N$ in $F=A^r$. 
\end{Corollary}

Mandal and Verma have recently proved that $e_1(A/Q) \leq 0$ for any parameter ideal $Q$ in $A$ (see \cite{MV}, and also \cite{GGHOPV}). Corollary \ref{e1} can be viewed as the module version of this fact. However, our proof based on the inequality in Theorem \ref{main} (1) is completely different from theirs and is considerably more simple. 

The proof of Theorem \ref{main} will be completed in section 3. It utilizes the fact that the Buchsbaum-Rim multiplicity of a parameter module 
can be determined as the Euler-Poincar\'e characteristic of the corresponding Eagon-Northcott complex. The next section is of preliminary character. In section 3, we will obtain Theorem \ref{main} as a corollary of a more general result (Theorem \ref{main2}).

%%%%%%%%%%%%%%%%%%%%%%%%%%%%%%%%%%%%%%%%%%%%%%%%%%%%%
\section{Preliminaries}
%%%%%%%%%%%%%%%%%%%%%%%%%%%%%%%%%%%%%%%%%%%%%%%%%%%%%

Let $(A, \fkm)$ be a Noetherian local ring of dimension $d$. Let $F=A^r$ be a free module of rank $r>0$. Let $S=\S_A(F)$ be the symmetric algebra of $F$. Let $N$ be a parameter module in $F$, that is, $N$ is a submodule of $F$ satisfying the conditions: (i) $\ell_A(F/N) < \infty$, (ii) $N \subseteq \fkm F$, and (iii) $\mu_A(N)=d+r-1$. We put $n=d+r-1$. Let $N^{\nu}$ be the homogeneous component of degree $\nu$ of the graded subalgebra $\Rees(N)=\Im(\S_A(N) \to S)$ of $S$. Let $\tilde{N}=(c_{ij})$ be the matrix associated to a minimal free presentation 
$$A^n \stackrel{\tilde{N}}{\to} F \to F/N \to 0$$ of $F/N$. Let $I(N)$ be the $0$-th Fitting ideal of $F/N$, which is the ideal generated by the maximal minors of $\tilde{N}$. 
Let $X=(X_{ij})$ be a generic matrix of the same size $r \times n$. We denote by $I_s(X)$ the ideal in the polynomial ring $A[X]=A[X_{ij} \mid 1 \leq i \leq r, 1 \leq j \leq n]$ 
generated by the $s$-minors of $X$. Let $B=A[X]_{(\fkm, X)}$ be the ring localized at the graded maximal ideal $(\fkm, X)$ of $A[X]$. The substitution map $A[X] \to A$ where $X_{ij} \mapsto c_{ij}$ now induces a map $\varphi: B \to A$. We consider the ring $A$ as a $B$-algebra via the map $\varphi$. Let $$\fkb=\Ker \varphi=(X_{ij}-c_{ij} \mid 1 \leq i \leq r, 1 \leq j \leq n)B.$$ Set $G=B^r$, and let $L$ denote the submodule $\Im (B^n \stackrel{X}{\to} G)$ of $G$. Let $G_{\nu}$ and $L^{\nu}$ be the homogeneous components of degree $\nu$ of the graded algebras $\S_B(G)$ and $\Rees(L)$, respectively. 
In the sequel we will utilize the exact sequences
\begin{equation}\label{ex}
0 \to L^{\nu}G_{t}/L^{\nu+1}G_{t-1} \to G_{\nu+t}/L^{\nu+1}G_{t-1} \to G_{\nu+t}/L^{\nu}G_{t} \to 0 \tag{$\ast$}
\end{equation}
where $\nu, t \geq 0$. Here $LG_{-1}=I_r(X)B$ and $NS_{-1}=I(N)$. 

We recall the following fact from \cite{HH1}:

\begin{Proposition}\label{rees}
For any integer $t \geq 0$, the $B$-module $L^{\nu}G_t/L^{\nu+1}G_{t-1}$ is isomorphic to the direct sum of $\binom{\nu+n-1}{n-1}$ copies of $G_t/LG_{t-1}$ for all $\nu \geq 0$. That is, we have 
for all $\nu, t \geq 0$ an isomorphism of $B$-modules $$L^{\nu}G_t/L^{\nu+1}G_{t-1} \cong (G_t/LG_{t-1})^{\binom{\nu+n-1}{n-1}}.$$
 
\end{Proposition}

\begin{proof}
See \cite[Proposition 3.1]{HH1}. 
\end{proof}

\begin{Lemma}
For any integers $\nu, t \geq 0$, we have the following:
\begin{enumerate}
\item[$(1)$] $\left(G_{\nu+t}/L^{\nu+1}G_{t-1} \right) \otimes_B (B/\fkb) \cong S_{\nu+t}/N^{\nu+1}S_{t-1}$;
\item[$(2)$] $\Supp_B(G_{\nu+t}/L^{\nu+1}G_{t-1})=\Supp_B(B/I_r(X)B)$; 
\item[$(3)$] The ideal $\fkb$ is generated by a system of parameters of the module $G_{\nu+t}/L^{\nu+1}G_{t-1}$. 
\end{enumerate}
\end{Lemma}

\begin{proof} The first assertion is easy to see. Let us then verify the second one. It is well-known that $\sqrt{\Ann_B (G/L)} = \sqrt{\Fitt_0 (G/L)}$ (see \cite[(16.2) Proposition]{BV}). Since $\Fitt_0 (G/L) = I_r(X)B$, we have $\Supp_B (G/L) = \Supp_B (B/I_r(X)B)$. An easy localization argument gives $$\Supp_B(G_{\nu+t}/L^{\nu+1}G_{t-1})=\Supp_B (G/L)$$ for all $\nu \geq 0 $ and $t \geq 1$. It therefore remains to show that $$\Supp_B (G_{\nu}/I_r(X)L^{\nu}) = \Supp_B (B/I_r(X)B),$$ but this is easily checked by using the exact sequence (\ref{ex}) in the case $t=0$ combined with Proposition \ref{rees}. Thus the assertion (2) follows. 
In order to prove the third assertion, recall first that $\dim B/I_r(X)B=d+(n+1)(r-1)=rn$ (see \cite[(5.12) Corollary]{BV}). The assertion (3) then follows from (1), (2) and the fact that $\fkb$ is generated by $rn$ elements. 
\end{proof}

\begin{Lemma}\label{perfect}
For any integer $\nu \geq 0$, we have 
\begin{enumerate}
\item[$(1)$] $G_{\nu}/I_r(X)L^{\nu}$ and $G_{\nu+1}/L^{\nu+1}$ are perfect $B$-modules of grade $d$;
\item[$(2)$] $G_{\nu+t}/L^{\nu+1}G_{t-1}$ has finite projective dimension for all $t \geq 0$. 
\end{enumerate} 
\end{Lemma}

\begin{proof}
The claim concerning $G_{\nu+1}/L^{\nu+1}$ in (1) is already known by \cite[Corollary 3.2]{BE} (see also \cite[Proposition 3.3]{KN}). Consider the exact sequence (\ref{ex}) with $t=0$. Since $B/I_r(X)B$ is a perfect $B$-module of grade $d$ (see \cite[(2.8) Corollary]{BV}), Proposition \ref{rees} implies that so is $L_{\nu}/I_r(X)L^{\nu}$. It thus follows that $G_{\nu}/I_r(X)L^{\nu}$ is a perfect $B$-module of grade $d$ for all $\nu \geq 0$. This proves (1). We can then prove $(2)$ by induction on $t$ using the exact sequence (\ref{ex}) and Proposition \ref{rees}. 
\end{proof}

\begin{Proposition}
For any $\fkp \in \Min_B(B/I_r(X)B)$, the equality 
$$\ell_{B_{\fkp}} \left( (G_{\nu+t}/L^{\nu+1}G_{t-1})_{\fkp} \right) = \ell_{B_{\fkp}} \left( (B/I_r(X)B)_{\fkp} \right) \binom{\nu+d+r-1}{d+r-1 } $$
holds true for all integers $\nu \geq 0$ and $t \geq 0$. 
\end{Proposition}

\begin{proof}
Take $\fkp \in \Min_B(B/I_r(X)B)$ and fix an integer $\nu \geq 0$. We put $I_{\fkp}=I_r(X)B_{\fkp}$. 
We start with the case $t=1$. 
By Lemma \ref{perfect} (1), $\grade_B \fkp =d$, because perfect modules are grade unmixed. It now follows that $I_{r-1}(X)B \nsubseteq \fkp$. Indeed, if $I_{r-1}(X)B \subseteq 
\fkp$ , then \begin{eqnarray*}
d &=& \grade_B \fkp \\
&\geq& \grade_B I_{r-1}(X)B \\
&=& (n-(r-1)+1)(n-(n-1)+1) \\
&=& 2(d+1), 
\end{eqnarray*}
which is a contradiction. Consider now the following free presentation $$B_{\fkp}^n \stackrel{X}{\to} G_{\fkp} \to G_{\fkp}/L_{\fkp} \to 0$$ of $G_{\fkp}/L_{\fkp}$. Since $I_{r-1}(X)B \nsubseteq \fkp$, we may assume that after elementary row and column transformations over $B_{\fkp}$ the matrix $X$ has the form 
$$\left(
\begin{array}{ccc|ccc}
 & & & & &  \\ 
 & E_{r-1} & & & O &  \\ 
 & & & & &  \\ \hline
0 & \cdots & 0 & a_1 & \cdots & a_d \\ 
\end{array}
\right),  
$$
where $E_{r-1}$ is the identity matrix of size $r-1$.  
Let us fix a free basis $\{ t_1, \dots , t_r \}$ for $G_{\fkp}$. Then we have $I_{\fkp}=(a_1, \dots , a_d)B_{\fkp}$ and $\Rees(L_{\fkp}) \cong B_{\fkp}[t_1, \dots , t_{r-1}, I_{\fkp}t_r]$. Therefore, we get the following isomorphisms: 
\begin{eqnarray*}
\left( G_{\nu+1}/L^{\nu+1} \right)_{\fkp} &\cong& (G_{\fkp})_{\nu+1}/(L_{\fkp})^{\nu+1} \\
&\cong&  B_{\fkp}[t_1, \dots , t_{r-1},  t_r]_{\nu+1}/B_{\fkp}[t_1, \dots , t_{r-1}, I_{\fkp}t_r]_{\nu+1} \\
&\cong& \bigoplus_{\substack{i_1, \dots , i_r \geq 0, \\ i_1+ \cdots +i_r=\nu+1}} \left( B_{\fkp}/I_{\fkp}^{i_r} \right) t_1^{i_1} \cdots t_r^{i_r} \\
&=& \bigoplus_{i=1}^{\nu+1} \biggl( \bigoplus_{\substack{i_1, \dots , i_{r-1} \geq 0 \\ i_1+ \cdots +i_{r-1}=\nu+1-i}} \left( B_{\fkp}/I_{\fkp}^{i} \right) t_1^{i_1} \cdots t_{r-1}^{i_{r-1}} \biggr) t_r^i. 
\end{eqnarray*}
Notice that the ring $B_{\fkp}$ is Cohen-Macaulay and the system of generators $\{ a_1, \dots , a_d \}$ of $I_{\fkp}$ forms a regular sequence on $B_{\fkp}$. Hence we can compute the length 
in question as follows: 
\begin{eqnarray*}
\ell_{B_{\fkp}} \left( (G_{\nu+1}/L^{\nu+1})_{\fkp} \right) &=& \sum_{i=1}^{\nu+1} \rank_{B_{\fkp}}(B_{\fkp}[t_1, \dots , t_{r-1}]_{\nu+1-i}) \cdot \ell_{B_{\fkp}}(B_{\fkp}/I_{\fkp}^i) \\
 &=& \ell_{B_{\fkp}}(B_{\fkp}/I_{\fkp}) \sum_{i=1}^{\nu+1} \binom{\nu-i+r-1}{r-2} \binom{i+d-1}{d} \\
 &=& \ell_{B_{\fkp}}(B_{\fkp}/I_{\fkp}) \sum_{i=0}^{\nu} \binom{\nu-i+r-2}{r-2} \binom{i+d}{d} \\
 &=& \ell_{B_{\fkp}}(B_{\fkp}/I_{\fkp}) \binom{\nu+d+r-1}{d+r-1}. 
\end{eqnarray*} 

We will next prove the case $t=0$. Consider the exact sequence (\ref{ex}) localized at $\fkp$. 
By Proposition \ref{rees} and the case $t=1$, we get 
\begin{eqnarray*}
\ell_{B_{\fkp}} \left( (G_{\nu}/I_r(X)L^{\nu})_{\fkp} \right) &=& \ell_{B_{\fkp}} \left( (G_{\nu}/L^{\nu})_{\fkp} \right) + \ell_{B_{\fkp}} \left( (L^{\nu}/I_r(X)L^{\nu})_{\fkp} \right) \\
 &=& \ell_{B_{\fkp}} (B_{\fkp}/I_{\fkp}) \left\{ \binom{\nu-1+n}{n} + \binom{\nu+n-1}{n-1} \right\} \\
 &=& \ell_{B_{\fkp}} (B_{\fkp}/I_{\fkp}) \binom{\nu+n}{n}. 
\end{eqnarray*}

The cases $t=0,1$ have been thus proven, we will now proceed by induction on $t$. 
Again, look at the exact sequence (\ref{ex}) localized at $\fkp$. By Proposition \ref{rees} and the induction hypothesis, we then have 
\begin{eqnarray*}
 & & \ell_{B_{\fkp}} \left( ( G_{\nu+t}/L^{\nu+1}G_{t-1} )_{\fkp} \right) \\  
& = & \ell_{B_{\fkp}} \left( ( G_{\nu+t}/L^{\nu+2}G_{t-2} )_{\fkp} \right) - 
\ell_{B_{\fkp}} \left( ( L^{\nu+1}G_{t-1}/L^{\nu+2}G_{t-2} )_{\fkp} \right) 
\\
& = & \ell_{B_{\fkp}} (B_{\fkp}/I_{\fkp}) \binom{\nu+1+n}{n} -
\ell_{B_{\fkp}} \left( (G_{t-1}/LG_{t-2})_{\fkp} \right) \binom{\nu+n}{n-1} \\
&=& \ell_{B_{\fkp}} (B_{\fkp}/I_{\fkp})  \left\{ \binom{\nu+n+1}{  n} - \binom{\nu+n}{ n-1} \right\} \\
&=& \ell_{B_{\fkp}} (B_{\fkp}/I_{\fkp})  \binom{\nu+n}{ n} 
\end{eqnarray*}
as desired. 
\end{proof}

%%%%%%%%%%%%%%%%%%%%%%%%%%%%%%%%%%%%%%%%%%%%%%%%%%%%%
\section{Proof of Theorem \ref{main}}
%%%%%%%%%%%%%%%%%%%%%%%%%%%%%%%%%%%%%%%%%%%%%%%%%%%%%

Theorem \ref{main} will be a consequence of the following more general result: 

\begin{Theorem}\label{main2}
Let $(A, \fkm)$ be a Noetherian local ring of dimension $d>0$. 
\begin{enumerate}
\item[$(1)$] For any rank $r>0$, the inequality 
$$\ell_A(S_{\nu+t}/N^{\nu+1}S_{t-1}) \geq e(F/N)\binom{\nu+d+r-1}{d+r-1} $$ 
always holds true for every parameter module $N$ in $F=A^r$ and all integers $\nu, t \geq 0$, where $NS_{-1}=I(N)$. 
\item[$(2)$] The following statements are equivalent: 
\begin{enumerate}
\item[$(i)$] $A$ is a Cohen-Macaulay local ring; 
\item[$(ii)$] There exists an integer $r>0$ and a parameter module $N$ of rank $r$ in $F=A^r$  such that the equality 
$$\ell_A(S_{\nu+t}/N^{\nu+1}S_{t-1}) = e(F/N)\binom{\nu+d+r-1}{ d+r-1} $$ 
holds true for some integers $0 \leq t \ (\leq d)$ and $\nu \geq 0$.
\end{enumerate}
\end{enumerate}
\end{Theorem}

In order to prove Theorem \ref{main2}, we need to introduce more notation. For any matrix $\fka$ of size $r \times n$ over an arbitrary ring, we denote by $K_{\bullet}(\fka)$ its Eagon-Northcott complex \cite{EN}. When $r=1$, the complex $K_{\bullet}(\fka)$ is just the ordinary Koszul complex of the sequence $\fka$. See \cite[Appendix A2]{E} for the definition and more details of complexes of this type. Recall in particular that if $N$ is a parameter module in a free module $F$ as in section 2, then $$e(F/N)=\chi(K_{\bullet}(\tilde{N})),$$ where $\chi(K_{\bullet}(\tilde{N}))$ denotes the Euler-Poincar\'e characteristic of the complex $K_{\bullet}(\tilde{N})$ (see~\cite{BR2} and~\cite{K}). 

\begin{Lemma}
Using the setting and notation of section 2, we have 
$$\chi(K_{\bullet}(\fkb) \otimes_B (B/I_r(X)B)) = \chi(K_{\bullet}(\tilde{N})).$$
\end{Lemma}

\begin{proof}
We use an idea from \cite{BR3}. Set $I=I_r(X)$. Since the Eagon-Northcott complex is compatible with the base change, $K_{\bullet}(X) \otimes_B A \cong K_{\bullet}(\tilde{N})$. The complex $K_{\bullet}(X)$ is a $B$-free resolution of $B/IB$ and hence, by tensoring with $A$ and taking the homology, we obtain  
$$H_p(K_{\bullet}(\tilde{N})) \cong  H_p(K_{\bullet}(X) \otimes_B A)  \cong  \Tor_p^B(B/IB, A)$$
for all $p \geq 0$. On the other hand, since the ideal $\fkb$ is generated by a regular sequence of length $rn$ in $B$, the ordinary Koszul complex $K_{\bullet}(\fkb)$ associated to a system of generators of ${\fkb}$ is a $B$-free resolution of $A$. Hence, by tensoring with $B/IB$, we can compute the $\Tor$ as follows: 
\begin{eqnarray*}
\Tor_p^B(B/IB, A) \cong H_p(K_{\bullet}(\fkb) \otimes_B (B/IB)). 
\end{eqnarray*}
Therefore, for any $p \geq 0$, $$H_p(K_{\bullet}(\tilde{N})) \cong H_p(K_{\bullet}(\fkb) \otimes_B (B/IB)). 
$$Thus $\chi(K_{\bullet}(\fkb) \otimes_B (B/IB)) = \chi(K_{\bullet}(\tilde{N}))$ as wanted.
\end{proof}

Now we can give the proof of Theorem \ref{main2}. 

\begin{proof}[Proof of Theorem \ref{main2}]
We use the same notation as in section 2. Put $I=I_r(X)$. 

(1): Fix integers $\nu \geq 0$ and $t \geq 0$. The ideal $\fkb$ being generated by a system of parameters of the module $G_{\nu+t}/L^{\nu+1}G_{t-1}$, we get
\begin{eqnarray*}
 & & \ell_A(S_{\nu+t}/N^{\nu+1}S_{t-1}) \\
&=& \ell_B((G_{\nu+t}/L^{\nu+1}G_{t-1}) \otimes_B(B/\fkb)) \\
&\geq& e(\fkb; G_{\nu+t}/L^{\nu+1}G_{t-1}) \\
&=& \sum_{ \fkp \in \Assh_B (G_{\nu+t}/L^{\nu+1}G_{t-1})} e(\fkb; B/\fkp) \cdot \ell_{B_\fkp} ((G_{\nu+t}/L^{\nu+1}G_{t-1})_{\fkp}) \\
&=& \sum_{\fkp \in \Assh_B (B/IB)} e(\fkb; B/\fkp) \cdot \ell_{B_{\fkp}} ((B/IB)_{\fkp}) \binom{\nu+d+r-1}{d+r-1} \\
&=& e(\fkb; B/IB) \binom{\nu+d+r-1}{d+r-1} \\
&=& \chi(K_{\bullet}(\fkb) \otimes_B (B/IB)) \binom{\nu+d+r-1}{ d+r-1} \\
&=& \chi(K_{\bullet}(\tilde{N}))\binom{\nu+d+r-1}{  d+r-1} \\
&=& e(F/N) \binom{\nu+d+r-1}{ d+r-1} 
\end{eqnarray*}
as desired, where $e(\fkb; \ast)$ denotes the multiplicity of $\ast$ with respect to $\fkb$.

(2): The other implication being clear, by the ideal case, for example, it is enough to show that (ii) implies (i). Assume thus that 
$$\ell_A(S_{\nu+t}/N^{\nu+1}S_{t-1})=e(F/N) \binom{\nu+d+r-1 }{ d+r-1}$$ for some $\nu \geq 0$ and $t \geq 0$. The above argument then gives $$\ell_B((G_{\nu+t}/L^{\nu+1}G_{t-1}) \otimes_B(B/\fkb)) = e(\fkb; G_{\nu+t}/L^{\nu+1}G_{t-1}).$$ 
It follows that $G_{\nu+t}/L^{\nu+1}G_{t-1}$ is a Cohen-Macaulay $B$-module of dimension $rn$. By Lemma \ref{perfect}, $G_{\nu+t}/L^{\nu+1}G_{t-1}$ is a $B$-module with finite projective dimension. Thus, by the Auslander-Buchsbaum formula,
\begin{eqnarray*}
\depth B &=& \depth_B(G_{\nu+t}/L^{\nu+1}G_{t-1}) + \pd_B (G_{\nu+t}/L^{\nu+1}G_{t-1}) \\
 &\geq & \dim_B(G_{\nu+t}/L^{\nu+1}G_{t-1}) +\grade_B (G_{\nu+t}/L^{\nu+1}G_{t-1}) \\
 &=& rn+d  \\
 &=& \dim B. 
\end{eqnarray*}
Therefore $B$ is Cohen-Macaulay so that $A$ is Cohen-Macaulay, too.  
\end{proof}

Taking $t=1$ in Theorem \ref{main2}, now readily gives Theorem \ref{main}. 

\begin{Remarks}
{\rm
Suppose that $A$ is Cohen-Macaulay. Because of Lemma \ref{perfect} (1) the above argument shows that the equality 
$$\ell_A(S_{\nu}/I(N)N^{\nu})=\ell_A(S_{\nu+1}/N^{\nu+1})=e(F/N) \binom{\nu+d+r-1 }{ d+r-1}$$ 
holds true for all $\nu \geq 0$. When $t \geq 2$, we do not know whether the Cohen-Macaulayness of $A$ implies the equality
$$\ell_A(S_{\nu+t}/N^{\nu+1}S_{t-1}) = e(F/N)\binom{\nu+d+r-1 }{ d+r-1},$$ 
except in the following cases:
\begin{enumerate}
\item[$(1)$] When $0 \leq t \leq d$, the equality $\ell_A(S_t/NS_{t-1})=e(F/N)$ holds true by \cite[Theorem 4.1]{HH2}. 
\item[$(2)$] When $d=2$, we know by \cite[Theorem 4.1]{HH1} that $G_{\nu+2}/L^{\nu+1}G_1$ is a perfect $B$-module of grade two for all $\nu \geq 0$. The argument in the proof of 
Theorem \ref{main2} then gives $$\ell_A(S_{\nu+2}/N^{\nu+1}S_1)=e(F/N)\binom{\nu+r+1}{r+1}$$ 
for all $\nu \geq 0$. 
\item[$(3)$] If $t \geq d+1$, then $\pd_B (G_t/LG_{t-1}) \geq d+1$ (see \cite[(2.19) Remarks]{BV}, for instance). So the equality does not hold in this case. 
\end{enumerate}
}
\end{Remarks}

%%%%%%%%%%%%%%%%%%%%%%%%%%%%%%%%%%%%%%%
%%%%%%%%%%%  References  %%%%%%%%%%%%%%
%%%%%%%%%%%%%%%%%%%%%%%%%%%%%%%%%%%%%%%

%%%%%%%%%%%%%%%%%%%%%%%%%%%%%%%%%%%%%%%%%
%%%%%%%%%%%%%%%%%%%%%%%%%%%%%%%%%%%%%%%%%
%%%%%%%%%%%%%%%%%%%%%%%%%%%%%%%%%%%%%%%%%


\begin{thebibliography}{99}

\bibitem{BUV}
J. Brennan, B. Ulrich and W. V. Vasconcelos, 
The Buchsbaum-Rim polynomial of a module, 
J. Algebra 241 (2001), 379--392

\bibitem{BV}
W. Bruns and U. Vetter, Determinantal Rings, Lecture Notes in Math.
1327, Springer-Verlag Berlin Heidelberg, 1988

\bibitem{BE}
D. A. Buchsbaum and D. Eisenbud, Generic free resolutions and a family of generically perfect ideals, Adv. in Math. 18 (1975), 245--301

\bibitem{BR2}
D. A. Buchsbaum and D. S. Rim, 
A generalized Koszul complex. II. Depth and multiplicity, 
Trans. Amer. Math. Soc. 111 (1964), 197--224

\bibitem{BR3}
D. A. Buchsbaum and D. S. Rim, 
A generalized Koszul complex. III. A Remark on Generic Acyclicity, 
Proc. Amer. Math. Soc. 16 (1965), 555--558

\bibitem{EN} 
J. A. Eagon and D. G. Northcott, 
Ideals defined by matrices and a certain complex associated with them, 
Proc. Roy. Soc. Ser. A 269 (1962), 188--204

\bibitem{E}
D. Eisenbud, 
Commutative algebra. With a view toward algebraic geometry, 
Graduate Texts in Mathematics, 150, Springer-Verlag, New York, 1995 

\bibitem{GGHOPV}
L. Ghezzi, S. Goto, J. Hong, K. Ozeki, T. T. Phuong, and W. V. Vasconcelos, 
Cohen-Macaulayness versus the vanishing of the first Hilbert coefficient of parameter ideals, to appear in J. London Math. Soc. 

\bibitem{HH1}
F. Hayasaka and E. Hyry, 
A family of graded modules associated to a module, 
Communications in Algebra, Volume 36 (2008), Issue 11, 4201--4217

\bibitem{HH2}
F. Hayasaka and E. Hyry, 
A note on the Buchsbaum-Rim multiplicity of a parameter module, Proc. Amer. Math. Soc. 138 (2010), 545--551

\bibitem{Ka}
Y. Kamoi, Remark on the polynomial type Poincar\'e series, Proceedings of The Second Japan-Vietnam Joint Seminar on Commutative Algebra (2006), 162--168

\bibitem{KN}
D. Katz and C. Naud\'e, 
Prime ideals associated to symmetric powers of a module, 
Comm. Algebra 23 (1995), no. 12, 4549--4555

\bibitem{K}
D. Kirby, 
On the Buchsbaum-Rim multiplicity associated with a matrix, 
J. London Math. Soc. (2) 32 (1985), no. 1, 57--61 

\bibitem{MV}
M. Mandal, B. Singh and J. K. Verma, On some conjectures about the Chern number of filtrations, Preprint 2010, arXiv:1001.2822

\end{thebibliography}
\end{document}